\documentclass[preprint, authoryear]{elsarticle}

\usepackage{amsmath}
\usepackage{amssymb}
\usepackage{amsthm}
\usepackage{amsfonts}
\usepackage{natbib}
\usepackage{float}
\usepackage{hyperref}
\usepackage{bm}
\usepackage{enumerate}
\usepackage{sectsty}
\usepackage{multirow}

\def\spacingset#1{\renewcommand{\baselinestretch}{#1}\small\normalsize}

\newtheorem{Lemma}{Lemma}
\newtheorem{Corollary}{Corollary}

\newtheorem{Theorem}{Theorem}
\newtheorem{Definition}{Definition}
\newtheorem{Remark}{Remark}

\DeclareMathOperator*{\Pois}{Pois}
\DeclareMathOperator*{\Var}{Var}

\newcommand{\XX}{X_{ \bm{1}  } }

\newcommand{\Xhar}[1]{X_{ \bm{#1} } }

\title{Asymptotics of multivariate contingency tables with fixed marginals}
\author[qz]{Quan~Zhou}
\ead{quan.zhou@rice.edu}
\address[qz]{Department of Statistics, Rice University, Houston, Texas 77005, U.S.A.}

\begin{document}
 
\begin{abstract}
We consider the asymptotic distribution of a cell in a $2 \times \cdots \times 2$ contingency table as the fixed marginal totals tend to infinity. 
The asymptotic order of the cell variance is derived and a useful diagnostic is given for determining whether the cell has a Poisson limit or a Gaussian limit. 
There are three forms of Poisson convergence. 
The exact form is shown to be determined by the growth rates of the two smallest marginal totals. 
The results are generalized to contingency tables with arbitrary sizes and are further complemented with concrete examples. 
\end{abstract}

\begin{keyword}
coupon collector's problem \sep negative association \sep negative relation \sep random allocation \sep Stein-Chen's method   
\end{keyword}

\maketitle
\vfill
\spacingset{1.1}

\newpage

\section{Introduction} \label{sec1}
This work considers the asymptotic distribution of a cell in a $2 \times \cdots \times 2$ contingency table as the fixed marginal totals tend to infinity. 
The literature on this problem has been documented under various names: ``the coupon collector problem,'' ``capture-recapture,'' ``the committee problem,'' ``matrix occupancy,'' ``random allocation,'' and ``allocation by complexes''~\citep[Sec.~6.4]{barbour1992poisson}. 
The reader is encouraged to consult \citet{holst1986birthday} and \citet{stadje1990collector} for historical accounting of these problems. 

The present work borrows the framework and terminology of the coupon-collector problem. Consider $n$ distinct coupons and $m$ coupon collectors operating independently and let the  $i$th collector collect $a_i$ distinct coupons. Let $\mathcal{C} = \{1, 2, \dots, m \} $ denote the set of the collectors. 
For each set $\mathcal{C}' \subseteq \mathcal{C}$, we are interested in the number of coupons that are collected by $\mathcal{C}'$ and by no others. 
These counts may be summarized in an $m$-way $2 \times 2 \times \cdots \times 2$ contingency table. 
Let $\Xhar{v}$ denote the count in the cell $\bm{v} = (v_1, \dots, v_m)$, where $v_i \in \{1, 2\}$   and $v_i = 1$ indicates that a coupon is collected by collector $i$. 
This contingency table must satisfy $\sum_{v_i = 1} \Xhar{v} = a_i$  and  $\sum_{v_i = 2} \Xhar{v} = n - a_i$, for $i = 1, \dots, m$, where the marginal total $a_i$ is treated as fixed. 
For the case of two collectors, the $2 \times 2$ contingency table is shown in Table~\ref{table:toy}. 

\begin{table}[h!]
\begin{center}
\begin{tabular}{ c c | c    c |}
& &  \multicolumn{2}{|c|}{Collector 2} \\
& & Collected &  Not collected  \\
\hline
\multirow{2}{*}{ Collector 1} & Collected    &  \multicolumn{1}{c|}{$X_{(1,1)}$}  & $X_{(1,2)}$ \\
\cline{3-4}
 & Not collected   &  \multicolumn{1}{c|}{$X_{(2,1)}$} & $X_{(2,2)}$ \\
\hline
\end{tabular}
\caption{The contingency table for $m = 2$. 
The cell counts must satisfy $X_{(1,1)} + X_{(1,2)} = a_1$ and $X_{(1,1)} + X_{(2,1)} = a_2$.
When we have a third collector, we can construct a $2 \times 2 \times 2$ contingency table by splitting each cell in the above table into two, according to whether the coupon is collected by collector $3$. 
}\label{table:toy}
\end{center}
\end{table}

We consider the distribution of an arbitrary cell under the following asymptotic conditions:\\
\vspace{-0.3cm}
\begin{enumerate}[({A}1)]
\item $n \rightarrow \infty$; 
\item $a_i = a_i (n) \rightarrow \infty$ and $n - a_i \rightarrow \infty$ for $i = 1, \dots, m$; 
\item $1 \leq a_1 \leq a_2 \leq \cdots \leq a_m \leq n - 1$; 
\item $a_i / n \rightarrow \alpha_i \in [0, 1]$ for $i = 1, \dots, m$. 
\end{enumerate}
Under (A1)--(A4), each cell can be treated equivalently up to relabelling of rows and columns. Therefore, without loss of generality, it suffices to consider one cell. Henceforth our analysis shall concern the cell $\XX$, where  $\bm{1} = (1, \dots, 1)$, i.e. the number of the coupons that are collected by all collectors.

To the best of our knowledge, the first complete analysis of all the possible asymptotic limits of $\XX$ is due to~\citet{vatutin1983limit}. 
The authors showed that $\XX$ has either a normal or a Poisson limit depending on whether $\Var(\XX)$ converges (see Theorem~\ref{th:main} below). 
This was accomplished by verifying that its generating function has only real roots~\citep[see also][]{kou1996asymptotics}.  
Alternative proofs for this problem and its variants are given in \citet[Chap.~VII]{kolchin1978random}, \citet{holst1980matrix}, \citet{mitwalli2002occupancy}, \citet{harris1989poisson},  and \citet{cekanavicius2000signed}.
See~\citet{smythe2011generalized} for an extension to the case in which $a_1, \dots, a_m$ are random. 
See ~\citet{lareida2017bittorrent} for a more recent application of these results.

\begin{Theorem}[\citet{vatutin1983limit}]\label{th:main}
Under the asymptotic assumptions {\normalfont (A1)--(A4)}, if $\Var(\XX) \rightarrow \infty$,  $\XX^* \equiv (\XX - \mathbb{E}(\XX) )/ \sqrt{\Var(\XX)}\overset{ \mathcal{D} }{\rightarrow}  \mathcal{N}(0, 1)$; 
if $\Var(\XX) \rightarrow \rho < \infty$,  $\XX$ has a Poisson limit in the sense that there exists a sequence of  constants $C_n$  such that $\XX + C_n  \overset{ \mathcal{D} }{\rightarrow} \Pois(\rho)   $ or $-\XX + C_n \overset{ \mathcal{D} }{\rightarrow} \Pois(\rho)$. 
(If $\rho = 0$, $\Pois(0)$ refers to the degenerate distribution $\delta_0$.)
\end{Theorem}

In Section~\ref{sec:order.V}, we calculate the asymptotic order of $\Var(\XX)$. This provides a useful diagnostic for determining whether the limiting distribution of $\XX$ given by Theorem~\ref{th:main} is normal or Poisson.  
In Section~\ref{sec:pois}, we show that the exact form of Poisson convergence is determined only by $a_1$ and $a_2$.
Section \ref{sec:general} generalizes the results of Sections \ref{sec:order.V} and \ref{sec:pois} to contingency tables of arbitrary size.

\section{Asymptotics of the cell variance}\label{sec:order.V}
Mathematical induction will be used to prove most of our key results. The induction setup is described as follows. 
In lieu of considering an $m$-way contingency table, we consider a sequence of contingency tables, each of which has a grand total count of $n$. 
The $k$th table records the coupon counts for the first $k$ coupon collectors and we use $\XX^{(k)}$ to denote the number of the coupons that are collected by each of the first $k$ collectors (whether the coupon is collected by the other collectors is not considered.) 
When we use induction to prove a statement regarding $\XX$, we always start from checking the statement for $\XX^{(2)}$ and then proceed to prove it for $\XX^{(k)}$, for $k = 2, \dots, m$. 

Let $E_k$ and $V_k$ denote the expectation and the variance of $ \XX^{(k)}$. 
Clearly, $E_1 = a_1$ and $V_1 = 0$, and for $k = 2, 3, \dots$, 
\begin{equation*}\label{eq:ek}
E_k   =   n \prod\limits_{i=1}^k \dfrac{ a_i }{n} = \dfrac{a_k}{n} E_{k - 1}, 
\end{equation*}
Since $\XX^{(2)}$ follows a hypergeometric distribution,
\begin{equation}\label{eq:var.m2}
V_2 = \dfrac{ a_1 a_2 (n - a_1) (n - a_2)  }{n^2(n-1)}.
\end{equation}
 
We proceed to derive a recursive characterization of $V_k$.  

\begin{Lemma} \label{lm:decomp}
For $k =2, 3, \dots$,  
\begin{equation}\label{eq:decomp2}
V_k =   \dfrac{a_k (n - a_k)  E_{k-1}  ( n - E_{k-1} )  }{n^2 (n-1)} +   \dfrac{ a_k (a_k - 1) }{n(n-1)}   V_{k-1} . 
\end{equation} 
\end{Lemma}

\begin{Remark}
The formula~\eqref{eq:decomp2} decomposes $V_k$ into two additive components. 
The first component is the variance of a cell from a $2 \times 2$ contingency table with fixed marginal totals $a_k$ and $E_{k-1}$. 
The second component captures the variation of $ \XX^{(k - 1)} $, which is 0 if $V_{k-1} \rightarrow 0$.  
If $\alpha_k = 1$, the second component converges to $V_{k-1}$. 
See~\citet{darroch1958multiple} for a closed-form expression for $V_k$.
\end{Remark}

\begin{proof}
By the law of total variance, we express $\Var(\XX^{(k)})$ as
\begin{equation*}\label{eq:decomp.V}
\begin{aligned}
V_k = \; & \mathbb{E} \left( \Var(  \XX^{(k)}    \mid \XX^{(k - 1)}  ) \right) 
+ \Var \left(\mathbb{E}(   \XX^{(k)}   \mid \XX^{(k - 1)}  ) \right).   
\end{aligned}
\end{equation*}
After conditioning on $\XX^{(k - 1)}$, $\XX^{(k )}$ is a hypergeometric random variable and thus we obtain  
\begin{align*}
\mathbb{E}(\XX^{(k )} \mid \XX^{(k - 1)}  )  = \; &      a_k   \XX^{(k - 1)} /n , \\
 \Var(  \XX^{(k )}  \mid \XX^{(k - 1)}  )  = \; &   a_k (n - a_k) \XX^{(k - 1)} (n  - \XX^{(k - 1)} )  / n^2 (n-1) . 
\end{align*}
Routine calculations using $ \mathbb{E}  (  \XX^{(k - 1)} )^2   = V_{k - 1} + E_{k-1}^2 $ yield~\eqref{eq:decomp2}. 
\end{proof}

Lemma~\ref{lm:decomp} will be important for proving a series of asymptotic results for our problem.  
Our first asymptotic result regards the asymptotic order of $\Var(\XX)$. 
Let $\sim$ denote the asymptotic equivalence, i.e., $x_n \sim y_n$ if $\lim_{n \rightarrow \infty} x_n/y_n = 1$. 
Let $\asymp$ denote that two positive sequences have the same asymptotic order, i.e., $x_n \asymp y_n$ if both $\limsup_{n \rightarrow \infty} x_n/y_n$ and $\liminf_{n \rightarrow \infty} x_n / y_n$  are finite and strictly positive. 
Hence, $x_n \sim y_n$ is a special case of $x_n \asymp y_n$. 

\begin{Theorem}[order of $\Var (\XX)$] \label{th:order.V}
Under the assumptions {\normalfont (A1)--(A4)}, the asymptotic order of $\Var(\XX)$ is 
\begin{equation*}
\Var(\XX) \asymp  \dfrac{  (n - a_1)   (n - a_2) }{n} \prod\limits_{i=1}^m \dfrac{a_i}{n} =  \left( 1 - \dfrac{a_1}{n} \right)  \left( 1 - \dfrac{a_2}{n} \right) \mathbb{E} (\XX). 
\end{equation*}
\end{Theorem}
 
\begin{Remark}
The claim is not true if $m \rightarrow \infty$. 
For example, let $a_i = n - 1$ for $ i = 1, \dots, m$ and $m^2/n \rightarrow 2\lambda$, 
we have $m - n + \XX \overset{\mathcal{D}}{\rightarrow} \mathrm{Pois}  (\lambda)$ and thus $\Var(\XX) \rightarrow \lambda$. 
This is in fact the classical birthday problem~\citep{arratia1989two,diaconis2002bayesian,dasgupta2005matching}.  
\end{Remark} 
 
\begin{proof}
We prove by induction on $\XX^{(k)}$. That is, we aim to prove that 
\begin{equation}\label{eq:th1}
\Var(\XX^{(k)}) = V_k \asymp  \dfrac{  (n - a_1)   (n - a_2) }{n} \prod\limits_{i=1}^k \dfrac{a_i}{n}  =  \left( 1 - \dfrac{a_1}{n} \right)  \left( 1 - \dfrac{a_2}{n} \right) E_k, 
\end{equation}
for $k = 2, \dots, m$. 
By \eqref{eq:var.m2}, the claim holds trivially for $\XX^{(2)}$.
We now suppose the above claim holds for  $\XX^{(k-1)}$ ($k \geq 3$) and consider $\XX^{(k)}$. 
 
The first subcase to consider is $\alpha_k = \lim a_k/n =  0$. 
In this subcase, by assumption (A3), $\alpha_i = 0$ for $i \leq k$. 
Hence, 
\begin{align*}
\left( 1 -  \dfrac{a_1}{n} \right)  \left( 1 - \dfrac{a_2}{n} \right) \sim 1. 
\end{align*}
Since $E_{k-1} \leq a_1$, $E_{k-1}/n \rightarrow 0$. 
Hence, the first component of $V_k$ in~\eqref{eq:decomp2} is
\begin{equation}\label{eq:induct}
\dfrac{a_k (n - a_k)  E_{k-1}  (n - E_{k-1} )   }{n^2 (n-1)}  \sim \dfrac{   a_k  E_{k-1}  }{n} = E_k. 
\end{equation}
According to the induction assumption, $V_{k-1} \asymp E_{k-1}$ and thus \eqref{eq:induct} has the same order as $ a_k  V_{k-1}/n$. 
Since $\alpha_k = 0$, the second component of $V_k$ in~\eqref{eq:decomp2} has a strictly smaller order. 
Hence, the order of $V_k$ is determined by its first component, which is asymptotically equal to $E_k$ by~\eqref{eq:induct}, and thus \eqref{eq:th1} holds.
 
The second subcase we consider is $ \alpha_k \in (0, 1]$. 
By the induction assumption, $ V_{k-1} \asymp (n - a_1) (n - a_2) E_{k-1} / n^2$.
Hence, 
\begin{equation*}
\dfrac{a_k (n - a_k)  E_{k-1}  ( n - E_{k-1} ) }{n^2 (n-1)}   \asymp \dfrac{ ( n  - a_k )  ( n -E_{k-1} )  V_{k-1} }{(n - a_1) (n - a_2) } <  \dfrac{   (n - E_{k-1} )  V_{k-1} }{ n - a_1  }. 
\end{equation*}
However, since $n - a_1 \leq n -  \XX^{(k-1)}  \leq    (k - 1) (n - a_1)$ and $k$ is finite, we have $n - E_{k-1} \asymp  n - a_1$.
Thus the first component of $V_k$ in~\eqref{eq:decomp2} has the same or a smaller order than $V_{k-1}$.
Since $\alpha_k > 0$ implies that the second component of $V_k$ in~\eqref{eq:decomp2} has the same asymptotic order as $V_{k-1}$, 
\begin{equation*}
V_k \asymp V_{k-1}  
\asymp  \left( 1 - \dfrac{a_1}{n} \right)  \left( 1 - \dfrac{a_2}{n} \right) E_{k - 1} 
\asymp  \left( 1 - \dfrac{a_1}{n} \right)  \left( 1 - \dfrac{a_2}{n} \right) E_k . 
\end{equation*} 
This completes the proof.
\end{proof}

By Theorem~\ref{th:main}, the limiting distribution of $\XX$ is fully determined by the convergence of the sequence $n^{- (m + 1)} (n - a_1) (n - a_2) a_1 \cdots a_m$. 
If it converges to zero, $\XX$ converges in probability to some constant; if it converges to some finite nonzero constant, $\XX$ has a Poisson limit. 
The following corollary shows that $\XX$ has a Poisson limit only when $\alpha_1, \alpha_2 \in \{0, 1\}$.

\begin{Corollary}\label{lm:conv}
Under the assumptions {\normalfont (A1)--(A4)}, $\Var(\XX)$ may converge to a finite constant only if $\alpha_1, \alpha_2 \in \{0, 1\}$ where $\alpha_i = \lim a_i/n$.  
This condition is necessary but not sufficient. 
\end{Corollary}

\begin{proof}
By assumption (A2), $a_i (n - a_i) / n \rightarrow \infty$ for every $i$. 
Hence, according to~\eqref{eq:var.m2}, the claim holds  for $\Var( \XX^{(2)} )$.  
Now consider $\Var( \XX^{(k)} ) $ with $k \geq 3$. By assumption (A3), if $\alpha_1$ or $\alpha_2$ is in $(0, 1)$, we have $\alpha_k > 0$. 
By Theorem~\ref{th:order.V}, this implies that $\Var(\XX^{(k)})$ has the same order as $\Var(\XX^{(k - 1)})$ and thus diverges.   
To see this condition is not sufficient, consider $\Var( \XX^{(2)} )$ for $a_1 = a_2 = n^{2/3}$, which implies $\alpha_1 = \alpha_2 = 0$.
A direct calculation using~\eqref{eq:var.m2} gives $\Var( \XX^{(2)} ) \sim n^{1/3}$. 
\end{proof}

\section{Poisson convergence}\label{sec:pois} 
Consider the simplest case of two coupon collectors and the associated $2 \times 2$ contingency table.  
If $\Var(\XX) \rightarrow 0$, the variance of any other cell must also tend towards zero since there is only one degree of freedom when the marginal totals are fixed. 
It is straightforward to see that $\XX$ should have three different ``limits''. 
First, if $\alpha_1 =  \alpha_2 = 0$, we have $\XX \rightarrow 0$  . 
Second, if $\alpha_1 = 0, \alpha_2 = 1$, then $X_{(1,2)} = a_1 - \XX \rightarrow 0$ , i.e. every coupon collected by the first collector would also be collected by the second. 
Third, if $\alpha_1 = \alpha_2 = 1$, then $X_{(2,2)} = \XX + n - a_1 - a_2 \rightarrow 0$, i.e.  no coupon would be missed by both collectors.

For the Poisson convergence of $m$-way $2 \times \cdots \times 2$ contingency table, it still suffices to consider the above three scenarios.

\begin{Lemma}\label{prop:main}
Under the assumptions {\normalfont (A1)--(A4)}, for $m \geq 2$,
\begin{enumerate}[(i)]
\item $a_1 / n \rightarrow 0, \, a_2 / n \rightarrow 0$: $ \mathbb{E}( \XX  )  \sim \Var(\XX) $; 
\item $a_1 / n \rightarrow 0, \, a_2 / n \rightarrow 1$: $ \mathbb{E}( a_1 - \XX ) \sim \Var(\XX) $; 
\item $a_1/n \rightarrow 1, \, a_2 / n\rightarrow 1$:  $ \mathbb{E} ( \XX  + (m - 1)n - \sum_{i=1}^m a_i ) 
\sim \Var(\XX)$. 
\end{enumerate}

\end{Lemma}
\begin{Remark}
No assumption about the convergence of $\Var(\XX)$ is needed. 
\end{Remark}

\begin{proof}
Just like we did in the proof for Theorem~\ref{th:order.V}, we prove each case separately by induction on the sequence $\XX^{(2)}, \dots, \XX^{(m)}$. 

\medskip 
\textbf{Case (i):}  
We  use induction to prove that, given $a_1 / n \rightarrow 0, \, a_2 / n \rightarrow 0$, we have $E_k \sim V_k$ for $k = 2, \dots, m$.  
For $k = 2$, the statement can be verified immediately using~\eqref{eq:var.m2}. 
Next, we assume $E_{k-1} \sim V_{k-1}$ and consider $E_k$ and $V_k$. 
Note that $\alpha_1 = \lim a_1 /n = 0$ implies $E_j / n \rightarrow 0$ for every $j$ since $E_j < a_1$. 
By the induction assumption and the identity $E_k = a_k E_{k - 1}/n$,  for the two components of $V_k$ in~\eqref{eq:decomp2}  we have
\begin{align*}
\dfrac{a_k (n - a_k)  E_{k-1}  ( n - E_{k-1} )  }{n^2 (n-1)} \sim \dfrac{E_k (n - a_k)}{n-1}, \quad
   \dfrac{ a_k (a_k - 1) }{n(n-1)}   V_{k-1} \sim \dfrac{ E_k (a_k - 1)}{ n - 1}. 
\end{align*}
It thus follows that $V_k = E_k + o(E_k)$.

\medskip 
\textbf{Case (ii):}
We  use induction to prove that, given $a_1 / n \rightarrow 0, \, a_2 / n \rightarrow 1$, we have $a_1 - E_k \sim V_k$ for $k = 2, \dots, m$.  
Again, for $k = 2$, the statement is immediate by~\eqref{eq:var.m2}. 
For the induction step, observe that $ (n - a_k) E_{k-1}/n = E_{k-1} - E_k $.
Hence, assuming $a_1 - E_{k - 1} \sim V_k-1$,  which is the induction assumption, and using the fact that $\alpha_k = 1$ and $E_{k-1}/n \rightarrow 0$, we obtain
\begin{align*}
\dfrac{a_k (n - a_k)  E_{k-1}  ( n - E_{k-1} )  }{n^2 (n-1)} \sim E_{k-1} - E_k, \quad
   \dfrac{ a_k (a_k - 1) }{n(n-1)}   V_{k-1} \sim a_1 - E_{k-1}. 
\end{align*}
Since both terms are always positive, by~\eqref{eq:decomp2}, we arrive at $V_k = a_1 - E_k + o(V_k)$. 

\medskip 
\textbf{Case (iii):}
We  use induction to prove that, given $a_1 / n \rightarrow 1, \, a_2 / n \rightarrow 1$, we have $E_k + (k - 1)n - \sum_{i=1}^k a_i \sim V_k$, for $k = 2, \dots, m$.  
For $k = 2$, the statement follows from~\eqref{eq:var.m2}. 
The induction argument is very similar to that of case (ii). 
We need only observe that $E_{k-1}/n \rightarrow 1$ and 
$ (n - a_k)  (n - E_{k-1} ) / n =  E_k   - E_{k - 1}  + n - a_k$.
\end{proof}

To establish the Poisson convergence of $\XX$, we use Stein-Chen's method for ``negatively associated''
and ``negatively related'' random variables, the definitions of which are given below. 

\begin{Definition}[\citet{joag1983negative}]
Random variables $Y_1, \dots, Y_N$ are said to be negatively associated if for every pair of disjoint subsets $A_1, A_2 \subseteq \{1,2, \dots, N\}$ and any nondecreasing functions $f_1$ and $f_2$, 
we have 
\begin{equation*}
\mathrm{Cov} (  f_1(Y_i, i \in A_1),   f_2 (Y_j, j \in A_2) ) \leq 0. 
\end{equation*}
\end{Definition}

\begin{Definition}[\citet{erhardsson2005stein}]
Bernoulli random variables $Y_1, \dots, Y_N$ are said to be negatively related if for each $i \in \{1, 2, \dots, N\}$ and any nondecreasing function $f: \{0, 1\}^{N-1} \mapsto \{0, 1\}$, we have 
\begin{equation*}
\mathbb{E} [ f(Y_1, \dots, Y_{i - 1},  Y_{i+1}, \dots, Y_N) \mid Y_i = 1  ] \leq \mathbb{E}[ f(Y_1, \dots, Y_{i - 1},  Y_{i+1}, \dots, Y_N) ]. 
\end{equation*}
\end{Definition}

In particular, negatively associated Bernoulli random variables are negatively related~\citep[Theorem 2.I]{barbour1992poisson}. 
We will first show that $\XX$,  $a_1 - \XX$ and $\XX  + (m - 1)n - \sum_{i=1}^m a_i $  can be decomposed into sums of negatively associated random variables. 
For $m = 2$,  all the three random variables follow hypergeometric distribution, and the negative association property of hypergeometric random variables has been well studied~\citep{joag1983negative, daly2012stein}.
Here we prove the general case $m \geq 2$.

\begin{Lemma}\label{prop:na}
$\XX$ and $a_1 - \XX$  can be written as sums of negatively related Bernoulli random variables. 
$\XX  + (m - 1)n - \sum_{i=1}^m a_i$ can be written as a sum of non-negative integer-valued negatively associated random variables.  
\end{Lemma}

\begin{proof}
For $\XX$, the statement was proven in \citet{barbour1989some} via coupling methods.
Here we use another method, which works for all three random variables. 
Let $I_{ij}$ ($i=1, \dots, m$,  $j = 1,\dots, n$) be a Bernoulli random variable such that $I_{ij} = 1$ if coupon $j$ is collected by the $i$th collector.  Let $J_{ij} = 1 - I_{ij}$. 
For each $i$, $\{I_{ij}: j = 1, \dots, n\}$ and $\{J_{ij}: j = 1, \dots, n\}$ are sets of negatively related random variables~\citep[Theorem 2.11]{joag1983negative}. 
The three random variables can be decomposed as 
\begin{equation}\label{eq:na.decomp}
\begin{array}{cl}
 \XX = \sum\limits_{j=1}^n Y_j ,  & Y_j \equiv \min ( I_{1j}, I_{2j}, \dots, I_{mj} )  ,     \\
  a_1 - \XX = \sum\limits_{j=1}^n Y'_j ,  &      Y'_j  \equiv  I_{1j}   \,   \max ( J_{2j}, \dots, J_{mj} )   ,   \\
 \XX  + (m - 1)n - \sum\limits_{i=1}^m a_i  = \sum\limits_{j=1}^n Y''_j , \quad  \quad &     Y''_j  \equiv  (  -1 + \sum_{i=1}^m J_{ij} )  \vee 0 . 
\end{array}
\end{equation}
All the three functions, $Y_j, Y'_j$ and $Y''_j$ are nondecreasing. 
Applying Property P6 and Property P7 of~\citet{joag1983negative} and using the independence assumption of the collectors, we see that $\{Y_j\}, \{Y'_j\}$ and $\{Y''_j\}$ are sets of negatively associated random variables.
Furthermore,  $\{Y_j\}$ and $\{Y'_j\}$ are negatively related since they are indicator random variables. 
\end{proof}

For a sum of negatively related random variables, Stein-Chen's method allows us to establish the Poisson convergence by simply comparing the first two moments.  

\begin{Theorem}[Poisson convergence of $\XX$]\label{th:poisson}
Under the assumptions {\normalfont (A1)--(A4)}, $\XX$ has a Poisson limit if and only if $\, \Var(\XX) \rightarrow \rho \in [0, \infty)$.  
($\Pois(0)$ refers to the degenerate distribution $\delta_0$.)
Let $\mathrm{Pois}(\rho)$ denote the Poisson distribution with parameter $\rho$. There are only three possible subcases:
\begin{enumerate}[(i)]
\item $a_1 / n \rightarrow 0, \, a_2 / n \rightarrow 0$: $ \XX   \overset{\mathcal{D}}{\rightarrow } \Pois(\rho) $. 
\item $a_1 / n \rightarrow 0, \, a_2 / n \rightarrow 1$: $ a_1 - \XX   \overset{\mathcal{D}}{\rightarrow } \Pois(\rho) $; 
\item $a_1/n \rightarrow 1, \, a_2 / n\rightarrow 1$:  $\XX  + (m - 1)n - \sum_{i=1}^m a_i   \overset{\mathcal{D}}{\rightarrow } \Pois(\rho)$.
\end{enumerate}
\end{Theorem}

\begin{proof}
We need only prove sufficiency. 
By Corollary~\ref{lm:conv}, the convergence of $\Var(\XX)$ requires $\alpha_1, \alpha_2 \in \{0, 1\}$. 
Since, by assumption (A3), $a_1 \leq a_2$,  Theorem~\ref{th:poisson} includes all the possible subcases where $\Var(\XX)$ converges. 
 By~\citet[Corollary 2.C.2]{barbour1992poisson}, if a random variable $Z$ is a sum of negatively related Bernoulli random variables, 
\begin{equation*}\label{eq:tv}
|| \mathcal{L}(Z) - \mathrm{Pois}(  \mathbb{E}(Z  ) ) ||_{\mathrm{TV}}  
< 1 -   \Var(Z) / \mathbb{E}(Z) , 
\end{equation*}
where $|| \cdot ||_{\mathrm{TV}}$ denotes the total variation distance.
Thus the Poisson convergence for case (i) and (ii) immediately follows from Lemma~\ref{prop:main} and Lemma~\ref{prop:na}. 

We now turn to case (iii).
To simplify notation, let $W \equiv   \XX  + (m - 1)n - \sum_{i=1}^m a_i$. 
Recall the decomposition $  W = \sum_{j=1}^n  Y''_j  $ given in~\eqref{eq:na.decomp}. 
Let $\theta \equiv \mathbb{E}( W   ) $ and $p \equiv  \theta^{-1} \sum_{j=1}^n \mathbb{P} (Y''_j = 1) $. 
By~\citet[Corollary 4.2]{daly2017relaxation}, 
\begin{equation}\label{eq:tv2}
|| \mathcal{L}( W ) - \mathrm{Pois}(\theta) ||_{\mathrm{TV}} 
\leq    1  + \theta  +  (1 - 2p) \left( \dfrac{ \Var(\XX) }{\theta}  + \theta \right).  
\end{equation}
By construction, for $k\geq 1$, $\mathbb{P} (Y''_j = k)$ is the probability that coupon $j$ is not collected by exactly $k+1$ collectors. 
Using the fact that $a_i/n \rightarrow 1$ for each $i$, we can show that for each $k' \geq 2$,   $ \mathbb{P} (Y''_j =  k') / \mathbb{P} (Y''_j = 1)  \rightarrow 0. $
This further implies that $\mathbb{E}( Y''_j )     \sim \mathbb{P} (Y''_j =  1) $ and thus $p \rightarrow 1$.  
Plugging this into~\eqref{eq:tv2} and using Lemma~\ref{prop:main}, we obtain  $|| \mathcal{L}( W ) - \mathrm{Pois}(\theta) ||_{\mathrm{TV}}  \leq o (\theta)  = o(1)$, which concludes the proof. 
\end{proof}

\section{Contingency tables with arbitrary sizes}\label{sec:general}
We now extend our results to a general $m$-way contingency table with size $r_1 \times r_2 \times \cdots \times r_m$. 
We use $\tilde{X}_{\bm{v}}$  to denote a cell in the general contingency table with position $\bm{v}= (v_1, v_2, \dots, v_m)$. 
The grand total of all the cells is  $n$. 
The marginal totals are fixed and are denoted by $b_i(j)$ ($i = 1,\dots, m$ and $j = 1, \dots, r_i$) which satisfy
\begin{equation}\label{eq:def.b}
b_i(j)  =  \sum\limits_{v_i = j}  \tilde{X}_{\bm{v}}  , \quad \quad  \sum\limits_{j=1}^{r_i} b_i(j) = n. 
\end{equation}
Note that the coupon collector's problem is a special case of the above with $r_i = 2$, $b_i(1) = a_i$ and $b_i(2) = n - a_i$ for each $i$. 
To study the asymptotic distribution of $\tilde{X}_{\bm{v}}$, we return to the coupon collector's model specified in Section \ref{sec1} and set
$a_i = b_i(v_i)$. 
Then  $\tilde{X}_{\bm{v}}$  has the same distribution as $\XX$ in the coupon collector model and its asymptotic distribution can be determined by Theorem~\ref{th:order.V} (after reordering $a_1, \dots, a_m$). 

We conclude the present work with two examples.  
First, consider a three-way contingency table with $r_1 = 3, r_2 = r_3 = 2$.  
The marginals are given by $\bm{b}_1 =  ( n^{1/4}, n^{1/2}, n - n^{1/4} - n^{1/2} )$, 
$\bm{b}_2 =  (n^{1/2}, n - n^{1/2})$ and $\bm{b}_3 = (n^{1/2}, n - n^{1/2})$ where $\bm{b}_i = (b_i(1), \dots, b_i(r_i))$. 
The limiting distributions of all the cells are given in Table~\ref{table:example}.
Using Theorem~\ref{th:main} and Lemma~\ref{prop:main}, each cell can be verified easily. 
It is also straightforward to check that all the marginal constraints are satisfied. 
Second, consider a three-way contingency table with the same size, same marginals $\bm{b}_1$ and $\bm{b}_2$, but $\bm{b}_3 = (n/2, n/2)$. 
The limiting distributions of all the cells are given in Table~\ref{table:example2}.
Now two thirds of the cells have normal limits and the variances of these cells are calculated manually.

\begin{table}[h!]
\begin{center}
\begin{tabular}{ |  c |  c    c | }
\hline 
$\tilde{X}_{ij1}$  & \hspace{2.2cm} $j=1$   \hspace{2.2cm} &  \hspace{2.2cm} $j=2$ \hspace{2.2cm} \\ 
\hline 
$i=1$  &   $ \tilde{X}_{111} \overset{\mathcal{P}}{\rightarrow} 0 $    &  $\tilde{X}_{121} \overset{\mathcal{P}}{\rightarrow} 0  $ \\ 
$i=2$  &   $\tilde{X}_{211} \overset{\mathcal{P}}{\rightarrow} 0  $  &   $ \tilde{X}_{221} \overset{\mathcal{D}}{\rightarrow} \Pois(1) $   \\
$i=3$  &  $ \tilde{X}_{311} \overset{\mathcal{D}}{\rightarrow} \Pois(1) $ & $ n^{1/2} - \tilde{X}_{321} \overset{\mathcal{D}}{\rightarrow} \Pois(2) $  \\
\hline 
\end{tabular}

\bigskip 

\begin{tabular}{ |  c |  c    c | }
\hline 
$\tilde{X}_{ij2}$  & \hspace{2.2cm} $j=1$   \hspace{2.2cm} &  \hspace{2.2cm} $j=2$ \hspace{2.2cm} \\ 
\hline 
$i=1$  &   $ \tilde{X}_{112} \overset{\mathcal{P}}{\rightarrow} 0 $    &  $n^{1/4} - \tilde{X}_{122}  \overset{\mathcal{P}}{\rightarrow} 0 $ \\ 
$i=2$  &   $\tilde{X}_{212} \overset{\mathcal{D}}{\rightarrow} \Pois(1)  $  &   $ n^{1/2} - \tilde{X}_{222} \overset{\mathcal{D}}{\rightarrow} \Pois(2) $   \\
$i=3$  &  $ n^{1/2} - \tilde{X}_{312} \overset{\mathcal{D}}{\rightarrow} \Pois(2) $  & $     \tilde{X}_{322} - n +n^{1/4} + 3n^{1/2}\overset{\mathcal{D}}{\rightarrow} \Pois(3) $  \\
\hline 
\end{tabular}
\caption{Example 1. The asymptotic distribution of a  $3 \times 2 \times 2$ contingency table with fixed marginals: $\bm{b}_1 =  ( n^{1/4}, n^{1/2}, n - n^{1/4} - n^{1/2} )$, 
$\bm{b}_2 =  (n^{1/2}, n - n^{1/2})$ and $\bm{b}_3 = (n^{1/2}, n - n^{1/2})$ where $b_i(j)$ is defined in~\eqref{eq:def.b}. 
}\label{table:example}
\end{center}
\end{table}

\begin{table}[h!]
\begin{center}
\begin{tabular}{ |  c |  c    c | }
\hline 
$\tilde{X}_{ij2}$  & \hspace{2.2cm} $j=1$   \hspace{2.2cm} &  \hspace{2.2cm} $j=2$ \hspace{2.2cm} \\ 
\hline 
$i=1$  &   $ \tilde{X}_{112} \overset{\mathcal{P}}{\rightarrow} 0 $    &  $2n^{-1/8} (\tilde{X}_{122}  - n^{1/4}/2 )  \overset{\mathcal{D}}{\rightarrow} \mathcal{N}^* $ \\ 
$i=2$  &   $\tilde{X}_{212} \overset{\mathcal{D}}{\rightarrow} \Pois(1/2)  $   &  $2n^{-1/4} (\tilde{X}_{222}  - n^{1/2}/2 )  \overset{\mathcal{D}}{\rightarrow} \mathcal{N}^* $  \\
$i=3$  &  $2n^{-1/4} (\tilde{X}_{312}  - n^{1/2}/2 )  \overset{\mathcal{D}}{\rightarrow} \mathcal{N}^* $  & 
$\dfrac{\tilde{X}_{322}  - n/2 + n^{1/2} + n^{1/4}/2 }{ n^{1/4}/\sqrt{2} }  \overset{\mathcal{D}}{\rightarrow} \mathcal{N}^* $ \\
\hline 
\end{tabular}
\caption{Example 2. The asymptotic distribution of a  $3 \times 2 \times 2$ contingency table with fixed marginals: $\bm{b}_1 =  ( n^{1/4}, n^{1/2}, n - n^{1/4} - n^{1/2} )$, 
$\bm{b}_2 =  (n^{1/2}, n - n^{1/2})$ and $\bm{b}_3 = (n/2, n/2)$ where $b_i(j)$ is defined in~\eqref{eq:def.b}. 
$\mathcal{N}^*$ denotes the standard normal distribution. 
Note that for any $i, j$, $\tilde{X}_{ij2}$ has the same distribution as $\tilde{X}_{ij1}$. 
}\label{table:example2}
\end{center}
\end{table}

\section*{Acknowledgments} 
The author is grateful to Dr. Philip Ernst and Dr. Larry Brown for their helpful discussions. 

%\section*{References} 

\bibliographystyle{plainnat}
\bibliography{ref}

\begin{thebibliography}{21}
\providecommand{\natexlab}[1]{#1}
\providecommand{\url}[1]{\texttt{#1}}
\expandafter\ifx\csname urlstyle\endcsname\relax
  \providecommand{\doi}[1]{doi: #1}\else
  \providecommand{\doi}{doi: \begingroup \urlstyle{rm}\Url}\fi

\bibitem[Arratia et~al.(1989)Arratia, Goldstein, and Gordon]{arratia1989two}
Richard Arratia, Larry Goldstein, and Louis Gordon.
\newblock Two moments suffice for poisson approximations: the chen-stein
  method.
\newblock \emph{The Annals of Probability}, pages 9--25, 1989.

\bibitem[Barbour and Holst(1989)]{barbour1989some}
A.~D. Barbour and Lars Holst.
\newblock Some applications of the stein-chen method for proving poisson
  convergence.
\newblock \emph{Advances in Applied Probability}, 21\penalty0 (1):\penalty0
  74--90, 1989.

\bibitem[Barbour et~al.(1992)Barbour, Holst, and Janson]{barbour1992poisson}
Andrew~D. Barbour, Lars Holst, and Svante Janson.
\newblock \emph{Poisson approximation}.
\newblock Clarendon Press Oxford, 1992.

\bibitem[Cekanavicius et~al.(2000)Cekanavicius, Kruopis,
  et~al.]{cekanavicius2000signed}
Vydas Cekanavicius, Julius Kruopis, et~al.
\newblock Signed poisson approximation: a possible alternative to normal and
  poisson laws.
\newblock \emph{Bernoulli}, 6\penalty0 (4):\penalty0 591--606, 2000.

\bibitem[Daly and Johnson(2017)]{daly2017relaxation}
Fraser Daly and Oliver Johnson.
\newblock Relaxation of monotone coupling conditions: Poisson approximation and
  beyond.
\newblock \emph{arXiv preprint arXiv:1706.04064}, 2017.

\bibitem[Daly et~al.(2012)Daly, Lef{\`e}vre, and Utev]{daly2012stein}
Fraser Daly, Claude Lef{\`e}vre, and Sergey Utev.
\newblock Stein's method and stochastic orderings.
\newblock \emph{Advances in Applied Probability}, 44\penalty0 (2):\penalty0
  343--372, 2012.

\bibitem[Darroch(1958)]{darroch1958multiple}
John~N. Darroch.
\newblock The multiple-recapture census: I. estimation of a closed population.
\newblock \emph{Biometrika}, 45\penalty0 (3/4):\penalty0 343--359, 1958.

\bibitem[DasGupta(2005)]{dasgupta2005matching}
Anirban DasGupta.
\newblock The matching, birthday and the strong birthday problem: a
  contemporary review.
\newblock \emph{Journal of Statistical Planning and Inference}, 130\penalty0
  (1):\penalty0 377--389, 2005.

\bibitem[Diaconis and Holmes(2002)]{diaconis2002bayesian}
Persi Diaconis and Susan Holmes.
\newblock A bayesian peek into feller volume i.
\newblock \emph{Sankhy{\=a}: The Indian Journal of Statistics, Series A}, pages
  820--841, 2002.

\bibitem[Erhardsson(2005)]{erhardsson2005stein}
Torkel Erhardsson.
\newblock Stein’s method for poisson and compound poisson.
\newblock \emph{`An Introduction to Stein's Method'}, 4:\penalty0 61, 2005.

\bibitem[Harris(1989)]{harris1989poisson}
Bernard Harris.
\newblock Poisson limits for generalized random allocation problems.
\newblock \emph{Statistics \& Probability Letters}, 8\penalty0 (2):\penalty0
  123--127, 1989.

\bibitem[Holst(1980)]{holst1980matrix}
Lars Holst.
\newblock On matrix occupancy, committee, and capture-recapture problems.
\newblock \emph{Scandinavian Journal of Statistics}, pages 139--146, 1980.

\bibitem[Holst(1986)]{holst1986birthday}
Lars Holst.
\newblock On birthday, collectors', occupancy and other classical urn problems.
\newblock \emph{International Statistical Review/Revue Internationale de
  Statistique}, pages 15--27, 1986.

\bibitem[Joag-Dev and Proschan(1983)]{joag1983negative}
Kumar Joag-Dev and Frank Proschan.
\newblock Negative association of random variables with applications.
\newblock \emph{The Annals of Statistics}, pages 286--295, 1983.

\bibitem[Kolchin et~al.(1978)Kolchin, Sevastyanov, and
  Chistyakov]{kolchin1978random}
Valentin~Fedorovich Kolchin, Boris~Aleksandrovich Sevastyanov, and
  Vladimir~Pavlovich Chistyakov.
\newblock Random allocations.
\newblock 1978.

\bibitem[Kou and Ying(1996)]{kou1996asymptotics}
S.~G. Kou and Z.~Ying.
\newblock Asymptotics for a 2$\times$ 2 table with fixed margins.
\newblock \emph{Statistica Sinica}, pages 809--829, 1996.

\bibitem[Lareida et~al.(2017)Lareida, Ho{\ss}feld, and
  Stiller]{lareida2017bittorrent}
Andri Lareida, Tobias Ho{\ss}feld, and Burkhard Stiller.
\newblock The bittorrent peer collector problem.
\newblock In \emph{Integrated Network and Service Management (IM), 2017
  IFIP/IEEE Symposium on}, pages 449--455. IEEE, 2017.

\bibitem[Mitwalli(2002)]{mitwalli2002occupancy}
Saleh~M. Mitwalli.
\newblock An occupancy problem with group drawings of different sizes.
\newblock \emph{Mathematica Slovaca}, 52\penalty0 (2):\penalty0 235--242, 2002.

\bibitem[Smythe(2011)]{smythe2011generalized}
R.~T. Smythe.
\newblock Generalized coupon collection: the superlinear case.
\newblock \emph{Journal of Applied Probability}, 48\penalty0 (1):\penalty0
  189--199, 2011.

\bibitem[Stadje(1990)]{stadje1990collector}
Wolfgang Stadje.
\newblock The collector's problem with group drawings.
\newblock \emph{Advances in Applied Probability}, 22\penalty0 (4):\penalty0
  866--882, 1990.

\bibitem[Vatutin and Mikhailov(1983)]{vatutin1983limit}
V.~A. Vatutin and V.~G. Mikhailov.
\newblock Limit theorems for the number of empty cells in an equiprobable
  scheme for group allocation of particles.
\newblock \emph{Theory of Probability \& Its Applications}, 27\penalty0
  (4):\penalty0 734--743, 1983.

\end{thebibliography}

\end{document}